\documentclass[12pt,oneside]{amsart}

\newif\ifpdfstoll
\ifx\pdfoutput\undefined
 \pdfstollfalse
\else
 \pdfoutput=1
 \pdfstolltrue
\fi

\ifpdfstoll
 \usepackage[pdftex,colorlinks=true,%
             pdftitle={The p-adic analytic subgroup theorem and applications},%
             pdfauthor={Tzanko Matev}]{hyperref}
\fi

\usepackage{amssymb} 
\usepackage{amsrefs}
\usepackage[all,cmtip]{xy}

\newcommand{\bQ}{\mathbb{Q}} 

\newcommand{\bP}{\mathbb{P}} 
\newcommand{\bZ}{\mathbb{Z}}
\newcommand{\bN}{\mathbb{N}}
\newcommand{\bC}{\mathbb{C}}
\newcommand{\bF}{\mathbb{F}}

\newcommand{\Qbar}{{\overline{\mathbb{Q}}}}
\newcommand{\cO}{\mathcal{O}}
\newcommand{\cD}{\mathcal{D}}
\newcommand{\cU}{\mathcal{U}}

\newcommand{\term}[1]{\textbf{#1}}
\newcommand{\abs}[1]{\lvert#1\rvert}
\newcommand{\norm}[1]{\lVert#1\rVert}
\newcommand{\pabs}[1]{\lvert#1\rvert_v}

\newcommand{\floor}[1]{\lfloor#1\rfloor}

\DeclareMathOperator{\ord}{ord}

\DeclareMathOperator{\Spec}{Spec}
\DeclareMathOperator{\Span}{span}

\newtheorem{thm}{Theorem} 
\newtheorem{prop}[thm]{Proposition}
\newtheorem{lemma}[thm]{Lemma}

 \theoremstyle{remark}

\newtheorem*{rem}{Remark}

\newcounter{saveenumi} 

\newcommand{\tang}[1]{t_{#1}}
\newcommand{\tG}{\tang{G}}
\newcommand{\tGM}{\tang{G_M}}
\newcommand{\tGK}{\tang{G_{K_v}}}
\newcommand{\OG}{{\cO_G}}
\newcommand{\shffrm}[2]{\Omega_{#1}^1(#2)}

\begin{document}
\frenchspacing 
\title{The p-adic analytic subgroup theorem and applications}
\author{Tzanko Matev}
\address{Department of Mathematics \\
         University of Bayreuth \\
         95440 Bayreuth, Germany}
\email{Tzanko.Matev@uni-bayreuth.de}
\date{October 15, 2010}

\begin{abstract}
 We prove a $p$-adic analogue of W\"ustholz's analytic subgroup
theorem. We apply this result to show that a curve embedded in 
its Jacobian intersects the $p$-adic closure of the Mordell-Weil 
group transversely whenever the latter has rank equal to 1. This 
allows us to give some theoretical justification to Chabauty 
techniques applied to finding rational points on curves whose
Jacobian has Mordel-Weil group of rank 1.
\end{abstract}

\maketitle

\section{Introduction}
In this paper we develop a $p$-adic analogue to W\"ustholz's analytic subgroup
theorem \cite{wustholz-89.2} and give an application to finding rational
points on curves. 

Let $G$ be a commutative algebraic group defined over the algebraic numbers
and let $V$ be a proper linear subspace of the Lie algebra of $G(\bC)$,
spanned by algebraic vectors. Then the analytic subgroup theorem states that
if $V$ contains a vector whose image under the exponential map is an algebraic
point in $G(\bC)$, then there exists an algebraic subgroup of $G$, containing
the said point, whose Lie
algebra is a linear subspace of $V$.

In order to state the $p$-adic analogue we replace the exponential map with
the logarithm, since it behaves better in the non-archimedean case. Let $V$ be
a linear subspace of the $p$-adic Lie algebra of $G$ spanned by algebraic
vectors. Then the theorem states that if the logarithm of an algebraic point
is contained in $V$, then there exists an algebraic subgroup, containing that
point, whose Lie algebra is a linear subspace of $V$. The precise statement is
given in Section 2.

The proof of the theorem is very similar to the proof of the original
result. One constructs a section of a sheaf on $G$ which has zeros of high
order on a certain finite set of algebraic points. We then proceed to show
that the negation of the theorem forces this section to have many more zeros,
which, using a multiplicity estimate, gives a contradiction. The main
difference to the original proof is the replacement of certain complex
analytic estimates with their $p$-adic analogues, which are stated in
Proposition~\ref{sec:upper-bound-6}. The proof of the theorem is presented in
Sections 4 to 8.

In Section 3 we give an application of the $p$-adic analytic subgroup theorem
to finding rational points on curves. Let $\varphi:C\to J$ be an embedding of
a curve defined over $\bQ$ into its Jacobian. We show that if the Jacobian is
simple and has Mordell-Weil rank equal to 1, then the $p$-adic completion of
$C$ intersects the $p$-adic closure of $J(\bQ)$ transversally at every point
in $C(\bQ)$. In particular, every rational point in $C$ has a $p$-adic
neighbourhood which intersects the $p$-adic closure of the Mordell-Weil group
at a single point.

Bruin and Stoll have combined the Mordell-Weil Sieve with Chabauty techniques
to develop an algorithm for finding the rational points on curves whose
Mordell-Weil rank is smaller than the genus (see \cite{stoll-bruin-09},
4.4). The proof that the algorithm terminates is conditional on Stoll's Main
Conjecture \cite{stoll-07}, as well as on an additional conjecture (Conjecture
4.2 in \cite{stoll-bruin-09}). Our results imply that when the rank of the
Mordell-Weil group is equal to 1, one can slightly modify their algorithm so
that its termination depends only on the Main Conjecture.

\section{Notation and Main Result}\label{sec:notation-main-result-1}
Let $\Qbar$ denote the field of algebraic numbers and let $M$ be a complete
ultrametric extension of $\Qbar$. Let $G$ be a commutative algebraic group
over $\Qbar$ and let $G_M$ be its base extension over $M$. If we denote by $\tG$
and $\tGM$ the tangent spaces at the identity element $e$ of $G$ and $G_M$
respectively, we have the relation $\tGM=\tG\otimes_\Qbar M$ and one has a
canonical $\Qbar$-linear embedding $\iota:\tG\to \tGM$, $v\mapsto v\otimes
1$. We will use this map to identify $\tG$ with a subset of $\tGM$, and to
identify the space $V\otimes_\Qbar M$ with a subspace of $\tGM$ for every
vector space $V\subseteq \tG$. We shall call vectors that lie in the image of
$\iota$ \term{algebraic vectors}.

Following Bourbaki~\cite{bourbaki-lie-I}*{\S3.7.6} one can define a logarithm
map $\log: G(M)^*\to \tGM$, where $G(M)^*$ is the set of all points $\gamma\in
G(M)$ with the property that the identity $e$ is an accumulation point of the
set $\{\gamma^n\ |\ n\in\bN\}$. Then we have

\begin{thm}\label{sec:main-theorem}
  Assume that $\gamma\in G(M)^*$ is an algebraic point and that $V\subseteq
  \tG$ is a $\Qbar$-linear subspace such that $\log(\gamma)\in
  V\otimes_{\Qbar}M$. Then there exists an algebraic subgroup $B\subseteq G$
  defined over $\Qbar$, such that $\gamma\in B(\Qbar)$ and $t_B \subseteq V$.
\end{thm}

\begin{rem}\label{sec:notation-main-result} 
  There are a couple of cases, where the statement of the theorem is trivial. If
  $V=\tG$, then one can simply choose $B=G$. If on the other hand $\gamma$ is
  a torsion point, say $\gamma^m=e$ for some $m\in \bN$, then we can pick $B$
  to be the group of $m$-torsion points. Therefore the theorem gives
  non-trivial implications only when $V$ is a proper linear subspace, and
  $\gamma$ is a non-torsion point.
\end{rem}

\section{Finding points on curves}
Let $K$ be a number field with a non-archimedean valuation $v$, and let $K_v$
be the completion with respect to this valuation. Let $K_v^{alg}$ denote the
field of all algebraic numbers in $K_v$. We are going to apply
Theorem~\ref{sec:main-theorem} in the case when the group $G$ is a simple
abelian variety defined over $K$ and the linear space $V$ is
one-dimensional. Then the set $G(K_v)$ is compact, which implies that
$G(K_v)^*=G(K_v)$ and that the logarithm is globally defined. We have the
following:
\begin{lemma} \label{sec:find-points-curv} Let $A$ be a simple abelian variety
  defined over $K$. Let $b\in t_A$ be a non-zero tangent vector and let
  $\gamma\in A(K)$ be a $K$-rational point of infinite order. Then the vectors
  $b$ and $\log\gamma$ are linearly independent over $K_v$.
\end{lemma}
\begin{rem}
  We should note that if the abelian variety $A$ is absolutely simple then the lemma
  is a trivial corollary of Theorem~\ref{sec:main-theorem}. The difficulty
  lies in showing the result for a simple, but not absolutely simple, abelian variety.
\end{rem}
\begin{proof}
  We shall give a proof by contradiction. Let $A$ be a simple abelian variety
  which is not absolutely simple. Assume that there exists an algebraic
  vector $b$ and a $K$-rational point $\gamma$ such that $\log\gamma$ lies
  in the one-dimensional $K_v$-linear space spanned by $b$. Then
  Theorem~\ref{sec:main-theorem} implies that there exists an algebraic
  subgroup $B$ defined over some finite Galois extension $L$, such that $\gamma\in
  B(L)$ and such that $t_B$ is a subspace of the $L$-linear space spanned by
  $b$. Let $B_0$ be the connected component at identity of this group, and let
  $\gamma_0$ be a multiple of $\gamma$, which lies in $B_0$. The group $B$ is
  one-dimensional, therefore $B_0$ is an elliptic curve. We will show that $B_0$ is
  defined over $K$ and thus derive a contradiction with the assumption that
  $A$ is simple.

  The idea of the following was suggested to me by Brendan Creutz. Let
  $\sigma\in Gal(L/K)$. The set of points $\{\gamma_0^n\colon n\in\bZ\}$ is
  infinite, therefore it is Zariski dense in $B_0$. The Galois automorphism
  induces a morphism $\sigma: A_L\to A_L$. This morphism is both open and
  continuous in the Zariski topology. It fixes the points $\gamma_0^n$, for all
  $n\in\bZ$, therefore it fixes their Zariski closure. But the closure of the set
  $\{\gamma_0^n\colon n\in\bZ\}$ is the curve $B_0$, therefore the morphism
  $\sigma$ fixes $B_0$. Since this is true for every Galois automorphism in
  $Gal(L/K)$, we have that $B_0$ is defined over $K$.
\end{proof}
\pagebreak
Let $C$ be a smooth curve defined over $K$. We assume that $C$ has at least
one $K$-rational point $P$. Then we have an embedding defined over~$K$
\[
\varphi_P :C \longrightarrow J
\]
into the Jacobian $J/K$ of $C$ such that $\varphi_P(P) = e$, where $e\in J(K)$
is the identity element. Let $A/K$ be a simple abelian variety,
and let
\[
i: C\longrightarrow A
\]
be a smooth non-constant morphism defined over $K$ which factors through
$\varphi_P$. (Since $A$ is simple this implies that the implied map $J\to A$ is
surjective.)

For any field extension $L$ of $K$, let $\shffrm{}{L}$ denote the sheaf of
algebraic 1-forms with coefficients in $L$ on a variety.  The general
theory of abelian varieties allows us to identify the cotangent space at the
identity $t^*_A\otimes L$ with the space of global 1-forms
$\Gamma(A,\Omega^1(L))$. We shall use that identification without further
mention.

\begin{thm}\label{sec:an-application}
  Assume that $W:= \Span_{K_v}\log A(K)$ is a 1-dimensional $K_v$-vector
  space. Then for every point $Q\in C(K_v^{alg})$ there exists a 1-form $w\in
  W^{\bot}\subset \Gamma(A,\Omega^1(K_v))$ such that $i^{*}(w)(Q) \neq 0$.
\end{thm}
In other words, the image of the curve $C$ under the map $i$ is transversal to
the $v$-adic closure of $A(K)$ at every intersection point which is
algebraic. The theorem, however, does not say anything for any possible
transcendental intersection points.
\begin{proof}
  Without loss of generality we can assume that $Q\in C(K)$ (otherwise we
  replace $K$ by a finite extension whose $v$-adic completion is still
  $K_v$). Let $t_{C,Q}$ denote the tangent space at $Q$, which is a
  1-dimensional $K$-vector space. We have a map $i_*:t_{C,Q}\to t_{A,i(Q)}$,
  where, similarly, $t_{A,i(Q)}$ is the tangent space at $i(Q)$. Composing
  this map with the differential of the translation map we get a $K$-linear
  homomorphism $\phi:t_{C,Q}\to t_A$. Note that, since $i$ is smooth, the map
  $\phi$ is an injection. Let $V:=\phi(t_{C,Q})$. Let $\gamma\in A(K)$ be a
  point of infinite order. Since $A$ is simple,
  Lemma~\ref{sec:find-points-curv} implies that $\log \gamma\not\in
  V\otimes_KK_v$, therefore the spaces $V\otimes_KK_v$ and $W$ are linearly
  independent over $K_v$. Hence there exists a global 1-form $w\in
  \Gamma(A,\Omega^1(K_v))$ which does not vanish on $V$ and such that
  $w(W)=0$. It is easily seen that any such form $w$ satisfies 
  $i^*(w)(Q)\neq 0$.
\end{proof}

\subsection{The modified Bruin-Stoll algorithm}
Let $C/\bQ$ be a smooth curve of genus at least two. We assume that $C$ has a
rational point $P$ which defines the embedding $\varphi_P:C\to J$ into the
Jacobian of $C$. We also assume that its Jacobian is simple, and
that the group $J(\bQ)$ has rank 1. We pick a prime $p$ such that the map
$\varphi_P$ can be extended to a smooth morphism between smooth and proper
$\bZ_p$-schemes \mbox{$\Phi_P:{\mathcal C}\to {\mathcal J}$}, where
$C_{\bQ_p}:=C\times_{\bQ}\Spec \bQ_p$ and $J_{\bQ_p}:=J\times_\bQ \Spec \bQ_p$
are the generic fibres of ${\mathcal C}$ and ${\mathcal J}$ respectively.

Let $Q\in C(\bQ)$. Then, according to Theorem~\ref{sec:an-application}, there
exists a global $p$-adic 1-form $w\in \Gamma(C,\Omega^1(\bQ_p))$, such that
its corresponding 1-form $\varphi_P^{*-1}w$ on $J_{\bQ_p}$ annihilates $\log
J(\bQ)$. We fix one such form. Let $t$ be a local parameter on $P$ which gives
a local parameter $\bar t$ on the reduction of $C$ as well. Then one has the
expansion $w = (a_0 + a_1t + \dots)dt$, where one can assume that $a_i\in
\bZ_p$. We define $v(w):= v(a_0)$. This definition does not depend on the
choice of the local parameter.

Let $J^1(\bQ_p)$ be the kernel of the reduction map $J(\bQ_p)\to J(\bF_p)$,
and let $J^{n+1}(\bQ_p):= p^n J^1(\bQ_p)$, where we consider $J^1(\bQ_p)$ as a
formal group. We have maps $\varphi_P^n:C(\bQ_p)\to J(\bQ_p)/J^n(\bQ_p)$ defined
in the obvious way. Then we have
\begin{prop}\label{sec:modified-bruin-stoll}
  If $p\geq 3$ and $n\geq v(w)+1$, then the preimage of $\varphi_P^n(Q)$ contains a
  single rational point.
\end{prop}
\begin{proof}
  The proof is a slight modification of the proof of Proposition 6.3 in
  \cite{stoll-06}. Let $r_1,\dots,r_g$ be local parameters of $J$ at
  $\varphi_P(Q)$ such that their reduction gives local parameters on the
  special fiber of ${\mathcal J}$. Then there exists an index $i$ such that
  $\varphi_P^*r_i$ and its reduction give local parameters on $C$ and ${\mathcal
    C}\times\Spec\bF_p$ respectively. Without loss of generality we can assume
  that $i=1$. Then one has the representation
  \[
  w = (a_0 + a_1r_1 +\cdots)dr_1.
  \]
  Since the residue class of $\varphi^n_P(Q)$ consists precisely of the points
  for which $\max\{\abs{r_1},\dots,\abs{r_g}\}\leq p^{-n}$, we have that
  $\abs{r_1}\leq p^{-n}$ for all points in the preimage of
  $\varphi^n_P(Q)$. Let $r:= p^{-n}r_1$.
  
  Then, the logarithm corresponding to $w$ is given by
  \[
  \lambda_w(r) = a_0p^nr + \frac{a_1p^{2n}}{2}r^2 +\dots +
  \frac{a_mp^{n(m+1)}}{m+1}r^{m+1} +\cdots
  \]
  According to the Chabauty theory (see \cite{stoll-06} for details) we know that
  all the rational points lying in the preimage of $\varphi_P^n(Q)$ correspond to
  zeros of $\lambda_w(r)$ for $\abs{r}\leq 1$. It is clear that if
  $v(\frac{a_mp^{n(m+1)}}{m+1})>v(p^na_0)$ for all $m\geq 1$, then this function will
  have a single solution $r=0$. One can easily check that for $p\geq 3$ and
  $n\geq v(a_0)+1$ this is precisely the case. This proves the
  proposition.
\end{proof}

The Mordell-Weil Sieve involves constructing a certain finite abelian group
$G$ together with a subset $X_G\subset G$ and a group homomorphism
\mbox{$\psi: J(\bQ)\to G$} such that $\varphi_P(C(\bQ)) \subseteq
\psi^{-1}(X_G)$. It is expected that one can construct $G$ in such a way so
that the last relation becomes an equality of sets (see \cite{stoll-bruin-09}
for details). In practice one picks a number $N$ which kills all elements in
$G$, and only considers the quotient map $\psi_N: J(\bQ)/NJ(\bQ)\to G$. Then
we have the following algorithm for finding the rational points on $C$:

\begin{enumerate}
  \item Fix a prime of smooth reduction $p\geq 3$, and compute $G_p:=J(\bF_p)$, as well
    as $X_p:= \varphi_P(C(\bF_p))$. Compute $G$ and $X_G$ using the Mordell-Weil
    Sieve and pick $N$ to be a multiple of the exponent of $G_p\times G$.
  \item Find all the residue classes of $J(\bQ)/NJ(\bQ)$ which map into
    $X_p\times X_G\subset G_p\times G$. If there are no such classes then terminate.
  \item For each of the residue classes we have found in Step 2, find the
    representative with smallest canonical height in $J(\bQ)$ and check if it
    comes from a rational point on the curve. Let $A$ be the set of all
    rational points which we have found in this way.
  \item For each rational point $Q\in A$, find a
    global analytic 1-form $w_Q$ which does not vanish on $Q$ but such that
    $\varphi_P^{*-1}w$ vanishes on $J(\bQ)$. Compute $v_Q:=v(w_Q)$. Let $n$ be
    the maximum of all such $v_Q$s.
  \item Set $G_p:=J(\bQ_p)/J^n(\bQ_p)$, $X_p := \varphi_P^n(C(\bQ_p))
    \setminus \varphi_P^n(A)$. Fix a new choice for $G$, $X_G$ and
    $N$ such that $N$ is a multiple of the exponent of $G_p\times G$.
  \item Go to Step 2.
\end{enumerate}

Proposition~\ref{sec:modified-bruin-stoll} guarantees that the only rational
points of $C$ mapped to $\varphi_P^n(A)$ are those coming from $A$,
which implies that if the algorithm terminates, then the union of all sets
$A$ which we find in Step 3 will be equal to $C(\bQ)$. This algorithm
is in practice equivalent to the one given by Bruin and Stoll. Theorem~\ref{sec:an-application}, however, guarantees that one can always perform Step
4, hence the termination of the algorithm depends only on Stoll's Main
Conjecture.

\section{Reduction to the semistable case}
We are going to prove Theorem~\ref{sec:main-theorem} by reducing it to a
special case. We will need the following definition. Let $G$ be a commutative
algebraic group of dimension $n$ defined over $\Qbar$, and let $V\subseteq
\tG$ be a $d$-dimensional $\Qbar$-linear subspace (we allow $d=n$). We set
$\tau(G,V):=d/n$, if \mbox{$\dim G\geq 1$}, and $\tau(G,V):=1$, otherwise.  The pair
$(G,V)$ is called \term{semistable}, if for all proper quotients $\pi: G\to
G'$ we have
\[
\tau(G,V) \leq \tau(G',V'),
\]
where $V'=\pi_*(V)$. Since $\tau(G',V')$ can take only finitely many possible
values, it follows that every pair $(G,V)$ has a semistable quotient. It is
also easy to see that if $\tau(G,V)=0$ or $\tau(G,V)=1$, then the pair $(G,V)$
is semistable.

The following statement is a special case of Theorem~\ref{sec:main-theorem}.
\begin{thm}\label{sec:reduc-semist-pair}
  Let $V$ be a proper linear subspace of $\tG$, such that $(G,V)$
  is semistable. Then $\log\gamma\notin V\otimes_{\Qbar} M$ for any
  algebraic non-torsion point  $\gamma\in G(M)^*$.
\end{thm}
\begin{lemma}
  Theorem \ref{sec:reduc-semist-pair} implies Theorem \ref{sec:main-theorem}.
\end{lemma}
\begin{proof}
  Let $G$ be a commutative algebraic group and let $V$ be a linear subspace of
  $\tG$. We proceed by induction on $\dim G$. 

  If the pair $(G,V)$ is semistable, then Theorem~\ref{sec:reduc-semist-pair}
  together with the remark in Section \ref{sec:notation-main-result-1}
  trivially imply Theorem~\ref{sec:main-theorem}. In particular,
  Theorem~\ref{sec:main-theorem} is true whenever $\dim G\leq 1$.
  
  Assume $(G,V)$ is not semistable and that Theorem~\ref{sec:main-theorem} is
  true for all commutative algebraic groups with dimension less than $\dim G$.
  Let $\gamma\in G(M)^*$ be an algebraic point such that
  $\log\gamma\in V\otimes M$. Let $\pi: G\to G'$ be a quotient to a
  semistable pair $(G',V')$, where $V'=\pi_*(V)$. Then
  \[
  1 > \tau(G,V) > \tau(G',V'),
  \]
  hence $\dim G'>\dim V'$ and $V'$ is a proper linear subspace of
  $\tG$. Theorem~\ref{sec:reduc-semist-pair} implies that the only algebraic
  points in $G'(M)^*$, whose logarithm lies in $V'\otimes M$, are the torsion
  points. On the other hand $\pi(\gamma)\in G'(M)^*$, since $\pi$ is
  continuous. The commutativity of the diagram
  \[
  \xymatrix{G(M)^* \ar[r]^\pi \ar[d]^\log & G'(M)^* \ar[d]^\log\\
    t_G\otimes M \ar[r]^{\pi_*} & t_{G'}\otimes M
  }
  \]
  implies that $\log\pi(\gamma)\in V'\otimes M$, therefore, by the
  inductive hypothesis $\pi(\gamma)$ is a torsion point.  Let $k$ be the order
  of $\pi(\gamma)$, and let $\gamma'=k \gamma$. Then $\gamma'\in H(M)$, where
  $H:=\ker \pi$. If $\tang{H}\subseteq V$, then we pick the algebraic group
  $B$ consisiting of the points $B(\Qbar)=\{\theta\in G(\Qbar)\colon
  \theta^k\in H(\Qbar)\}$. Clearly $\gamma\in B(\Qbar)$ and $t_B=t_H\subseteq
  V$, hence Theorem~\ref{sec:main-theorem} is true. If on the other hand
  $t_H\not\subseteq V$, then,
  using the inductive hypothesis, we can apply Theorem~\ref{sec:main-theorem}
  to $H$, $\gamma'$ and the space $t_H\cap V$. Since $\log\gamma'\in
  (t_H\cap V)\otimes M$, there exists an algebraic subgroup $B_1$ of $H$, such
  that $t_{B_1}\subseteq t_H\cap V \subseteq V$ and $\gamma'\in
  B_1(\Qbar)$. We then pick $B$ such that $B(\Qbar)=\{\theta\in G(\Qbar)\colon
  \theta^k\in B_1(\Qbar)\}$. This group satisfies the properties prescribed in
  Theorem~\ref{sec:main-theorem}. This concludes the proof.
\end{proof}

\section{Preliminaries}\label{sec:preliminaries-1}
In this section we introduce some notation and standard results which are
going to be needed for the proof of Theorem~\ref{sec:reduc-semist-pair}.

We shall prove Theorem~\ref{sec:reduc-semist-pair} by contradiction.  We
assume that there exist a linear subspace $V\subsetneq t_G$ and an algebraic
non-torsion point $\gamma\in G(M)^*$ such that the pair $(G,V)$ is semistable
and $\log\gamma\in V\otimes M$.  We will denote $n=\dim G$, $d=\dim V$. Let
$\Gamma$ be the group generated by $\gamma$ in $G(M)$.

We fix an embedding $\phi: G\to \bP^N$ of $G$ into an $N$-dimensional
projective space such that $\phi(\Gamma) \cap \{X_N = 0\} = \emptyset$. (We
will use $X_0,\dots, X_N$ to denote the coordinates in $\bP^N$.) We can
always choose such an embedding. Indeed, let $\psi : G\to \bP^N$ be an
arbitrary embedding. Then the composition of $\psi$ with a projective
automorphism which sends some hyperplane defined over a sufficiently large
extension of the field of definition of $\Gamma$ to the hyperplane $\{X_N=0\}$
will give us the desired morphism $\phi$.  

We shall identify $G$ with $\phi(G)$. We denote $\bar{U}
:= \bar{G}\cap \{X_N\neq 0\}$, and $U:=\bar{U}\cap G$. Here $\bar{G}$ is the
Zariski closure of $G$ in $\bP^N$. The restriction morphism
$\cO_{\bar{G}}(\bar{U})\to \cO_{\bar{G}}(U) = \OG(U)$ is an injection. Therefore
we shall identify $\cO_{\bar{G}}(\bar{U})$ with a subset of $\OG(U)$.

The proof of the following lemma is given in \cite{wustholz-89.2}*{Section 2}
\begin{lemma}\label{sec:set-up-1}
  There exists a Zariski open set $U'\subset G\times G$ such that $\Gamma\times
  \Gamma\subset U'$ and such that the group law is represented on $U'$ by
  a single set of bi-homogeneous polynomials $E_0,\dots,E_N\in
  \Qbar[X_0,\dots, X_N,X'_0,\dots,X'_N]$ of bi-degree $b$, whose
  coefficients are algebraic integers.
\end{lemma}
We fix such a set of polynomials and denote $\bold{E}:=(E_0:\dots: E_N)$.
\begin{rem}
  One can show that for an appropriate embedding $\phi$ the group
  operation can be given by bi-homogenous polynomials of bi-degree 2. However,
  since our results are not effective, we are not going to need this fact.
\end{rem}

From now on we shall assume that $G$, $V$, $\gamma$ and $\bold{E}$ are defined
over a fixed number field $K$ with a fixed embedding into $\Qbar$. We will
abuse notation by denoting the tangent space at $e$ of $G$ again by $t_G$. It
is now a $n$-dimensional $K$-linear space. The embedding $K\hookrightarrow
\Qbar \hookrightarrow M$ gives rise to a non-archimedean valuation $v$ on
$K$. Let $K_v$ be the completion of $K$ with respect to that valuation. Then
$\log\gamma\in \tGK = \tG\otimes_{K}K_v$.

We take the height of a point in $\bP^N(\Qbar)$ to be its projective
height. This means that if $\alpha\in\bP^N(L)$ for some number field $L$,
$\alpha = (\alpha_0,\dots,\alpha_N)$, then\pagebreak
\[
h(\alpha) := \sum_{w\in M_L}\log(\max_i\abs{\alpha_i}_w),
\]
where the sum runs over the set $M_L$ consisting of all places of $L$.  Here we define
$\abs{x}_w:= \abs{N_{L_w/\bQ_{w'}}(x)}_{w'}^{1/[L:\bQ]}$, where $w'$ is the
unique place of $\bQ$ such that $w'|w$. The height $h(P)$ of a homogenous
polynomial $P$ of degree $D$ is the height of its coefficients taken as a
point in $\bP^A(\Qbar)$, where $A=\binom{N+D}{N}-1$.  The following lemma is
proved in \cite{serre-79}*{Proposition 5}:
\begin{lemma}\label{sec:preliminaries}
  Let $\alpha_1,\dots,\alpha_k\in G(K)$. Then there exist constants $c_1$,
  $c_2$ such that 
  \[
  h(\alpha_1^{n_1}\dots\alpha_k^{n_k}) \leq c_1 + c_2(\sum_i\abs{n_i})^2
  \text{, for all } (n_i)\in \bZ^k.
  \]
\end{lemma}
\section{Differentiation}
We fix a basis $\partial_1,\dots \partial_n$ of $\tG$.  It is a standard fact
in the theory of group varieties that for every $\partial\in \tG$ there exists
a unique translation invariant derivation $D(\partial)$. This means that we
have a morphism of sheaves of $K$-linear spaces $D(\partial):\OG\to\OG$ such
that if $\cU$ is a Zariski-open subset of $G$, $f,g\in \OG(\cU)$, $\alpha\in
\cU(\Qbar)$ and $c\in K$ then
\begin{enumerate}
  \item $D(\partial) c= 0$;
  \item $D(\partial) fg = f D(\partial) g + g D(\partial) f$;
  \item $D(\partial)f(\alpha) = D(\partial)(f\circ T_\alpha) (e) = \partial (f\circ T_\alpha)$,
    \setcounter{saveenumi}{\value{enumi}}
\end{enumerate}
where $T_\alpha(\beta):= \alpha\beta$ is the translation-by-$\alpha$
morphism. Since $G$ is commutative, we also have
\begin{enumerate}
\setcounter{enumi}{\value{saveenumi}}
\item $D(\partial)D(\partial') f=D(\partial')D(\partial)f$  
\end{enumerate}
for any two vectors  $\partial,\partial'\in\tG$. From now on we shall use the
same notation for a vector in $\tG$ and for its corresponding derivation. For
any linear space $W\subseteq t_G$, $K[W]$ will denote the ring of differential
operators with coefficients in $K$ generated by derivations in $W$.

Let $\Delta_1,\dots,\Delta_d$ be a basis of $V$. We denote
\[
\cD(\infty):=\{\Delta\in K[V]\colon
\Delta=\Delta_1^{t_1}\dots\Delta_d^{t_d},\ t_i\geq 0\}.
\]
For any differential operator $\Delta\in \cD(\infty)$, $\Delta =
\Delta_1^{t_1}\dots\Delta_d^{t_d}$, we denote $\abs{\Delta}:=t_1+\dots +
t_d$. For any non-negative integer $T$ we set
\[
\cD(T):= \{\Delta \colon \abs{\Delta}\leq T\}.
\] Despite the notation, the sets $\cD(\infty)$ and
$\cD(T)$ depend on the choice of basis of $V$. We shall later fix a convenient
basis to work with.

Let $\alpha\in G(K)$ and let $P\in K[X_0,\dots,X_N]$ be a homogenous
polynomial of degree $D$. Assume that $Q$ is another such polynomial with the
same degree and such that $Q(\alpha)\neq 0$.  We define the \term{order} of $P$
along $V$ at the point $\alpha$ to be the smallest number $t$ such that there
exists $\Delta\in \cD(t)$ with
\[
\Delta\frac{P}{Q}(\alpha) \neq 0
\]
We denote the order by $\ord_{V,\alpha}P$. One can check that this definition
depends neither on $Q$, nor on the choice of basis for $V$.

Let $P$ be a homogenous polynomial of degree $D$ and let $\Delta\in
K[V]$. Then $\frac{P\circ \bold{E}(X_0,\dots,
  X_N,Y_0,\dots,Y_N)}{Y_N^{bD}}$, when considered as a function of $\bold{Y}$,
induces a rational function $f(X_0,\dots,X_N)\in \OG(U)$. We define
\[
P_\Delta(X_0,\dots,X_N) := (\Delta f(X_0,\dots,X_N))(e)
\]
Then $P_\Delta$ is a homogenous polynomial of degree $bD$ and it is not
difficult to show that $\ord_{V,\alpha}P_\Delta \geq \ord_{V,\alpha}P -
\abs{\Delta}$. This allows us to study the multiplicity of $P$ at a point
using the polynomials $P_\Delta$.

If $P$ is any polynomial we write $\abs{P}_v$ for the maximum of the absolute
values of its coefficients. We shall later need an estimate of
$\pabs{P_\Delta}$ in terms of $\pabs{P}$. We are going to show next that
after an appropriate choice of basis for $V$ such an estimate becomes trivial.
More precisely, let us call a basis $\Delta_1,\dots,\Delta_d\in V$,
\term{nice}, if the following property holds: For any homogenous polynomial
$P$ with $\pabs{P}\leq 1$ and any $\Delta\in \cD(\infty)$ one has
$\pabs{P_\Delta}\leq 1$.
\begin{lemma}\label{sec:differentiation}
  The linear space $V$ has a nice basis.
\end{lemma}
\begin{proof}
  Let $\Delta'_1,\dots,\Delta'_d$ be any basis of $V$. The $K$-algebra
  $\OG(U)$ is finitely generated. Pick a set of generators $f_1,\dots,f_k\in
  \OG(U)$ which contains the functions $\frac{X_i}{X_N}$ for all
  $i=0,\dots,N-1$ and such that $\abs{f_j(e)}_v\leq 1$ for all $j=1,\dots,
  k$. Then there exist polynomials $F'_{ij}\in K[T_1,\dots, T_k]$ such that
  $\Delta'_if_j=F'_{ij}(f_1,\dots,f_k)$.  Pick $\Delta_i= c_i\Delta_i'$ where
  $c_i$ is any non-zero integer such that $\abs{c_iF'_{ij}}_v\leq 1$ for all
  $j$. We are going to show that this basis has the required property.
  
  Let $F\in K[T_1\dots,T_k]$ be such, that $\abs{F}_v\leq 1$. It is then easy
  to show that $\abs{\Delta F(f_1,\dots,f_k)(e)}_v\leq 1$ for any choice of
  $\Delta = \Delta_1^{t_1}\cdots \Delta_d^{t_d}$.

  Let now $P$ be a homogenous polynomial of degree $D$. Then we have the
  representation
  \[
  \begin{split}
  f(X_0,\dots,X_N) &:= \frac{P\circ E(X_0,\dots,X_N,Y_0,\dots,Y_N)}{Y_N^{bD}}
  \\
  &= \sum_JA_JX_0^{J_0}\dots X_N^{J_N},
  \end{split}
  \]
  where $A_J\in \OG(U)$. Since the coefficients of $E$ are algebraic integers
  and $\abs{P}_v\leq 1$, it follows that $f$, when considered as a polynomial
  in $X_0,\dots, X_N, \frac{Y_0}{Y_N},\dots, \frac{Y_{N-1}}{Y_N}$, has
  $\abs{f}_v \leq 1$. Hence, since the functions $Y_i/Y_{N}$ lie in the set
  $\{f_1,\dots,f_k\}$, one has a representation $A_J =
  \sum_{I}{A_J^I}f_1^{I_1}\dots f_k^{I_k}$, where $A_J^I\in K$,
  $\abs{A_J^{I}}_v\leq 1$. By the observation in the previous paragraph it follows
  that $\abs{\Delta A_J(e)}_v \leq 1$ for all $\Delta\in \cD(\infty)$.  Since
  \[
  P_\Delta(X_0,\dots,X_N) = \sum_{J} \Delta A_J(e) X_0^{J_0}\dots X_N^{J_N},
  \]
  one concludes that $\abs{P_\Delta}_v \leq 1$ for all $\Delta\in \cD(\infty)$.
\end{proof}

From now on we fix a nice basis $\Delta_1,\dots, \Delta_d$ of $V$.

In order to complete the proof of Theorem~\ref{sec:reduc-semist-pair} we will
need to apply a multiplicity estimate. We state here one estimate, given in
\cite{baker-wustholz-07}:
\begin{thm} \label{sec:differentiation-2} Let $(G,V)$ be a semistable pair and
  let $\gamma\in G(M)$. We fix an embedding of $G$ into projective space, and
  a Zariski open subset $U$ as in Section~\ref{sec:preliminaries-1}. Let
  $\gamma_0\in \Gamma$. We fix a basis $\Delta_1,\dots,\Delta_d$ of $V$. There
  exists an effectively computable constant $c>0$ with the following
  property. Let $S_0$, $T_0$ and $D_0$ be non-negative integers such that
  \[
  S_0T_0^d>cD_0^n.
  \]
  Assume in addition that there exists a homogenous polynomial $P$
  of degree $D_0$, which does not vanish identically on $G$ and such that
  \[
  P_{\Delta_1^{t_1}\dots\Delta_d^{t_d}}(\gamma_0^s) = 0
  \]
  for all $0\leq s\leq S_0$ and all non-negative integers $t_1,\dots,t_d$ with \\
  \mbox{$0\leq t_1+\dots+t_d\leq T_0$}. Then $\gamma_0$ is a torsion point.
\end{thm}
One could also use the more general Philippon multiplicity estimate
\cite{philippon-86} of which the previous theorem is an easy consequence (See
Chapter 11, Corollary 4.2 in \cite{nesterenko-philippon-01}).

\section{Proof of Theorem~\ref{sec:reduc-semist-pair}}
There exists a system of $v$-adic open neighbourhoods of $e$ in $G(K_v)$
consisting of subgroups of $G(K_v)$. Hence there exists an open (and closed)
subgroup $W_0\subset U(K_v)\cap \{\alpha\in \bP^N(K_v)\colon
\abs{\frac{X_i}{X_N}(\alpha)}_v\leq 1 \text{ for all } i=0,\dots,N-1\}$ where
the logarithm is a local isomorphism. Since $\gamma\in G(K_v)^*$, there exists
$k\in \bN$ such that $\gamma^k\in W_0$. Since $\log\gamma =
\frac{1}{k}\log(\gamma^k)$, in order to prove
Theorem~\ref{sec:reduc-semist-pair} one only needs to show that
$\log(\gamma^k)\not\in V\otimes K_v$. Therefore, without loss of generality,
we can assume that $\gamma\in W_0$.

Let $D$, $T$, $S$, $l$ be nonnegative integers. Since our result is not
effective we do not need to keep track of all the constants appearing in our
estimates. Therefore, in order to simplify notation,
we will use the letter $c$ to denote a sufficiently large positive number,
which can change from line to line, and which does not depend on $D$, $T$, $S$
and $l$.  Let $\gamma_1 := \gamma^p$, where $p$ is the
prime which is extended by $v$, and let $\gamma' = \gamma_1^l$. We denote the
group generated by $\gamma'$ by $\Gamma'$, and we also set $\Gamma'(S) :=
\{\gamma'^s\colon 0\leq s\leq S\}$.

The main steps of the proof are contained in the following three
propositions. We defer the proofs of those propositions to the next section.
\begin{prop}[Auxiliary polynomial]\label{sec:auxiliary-polynomial-1}
  Assume that the following inequality holds
  \begin{equation}\label{eq:2}
    D^n \geq 2n(\deg K)(T+n)^{d}(S+1)
  \end{equation}
  Then there exists a homogeneous polynomial $P$ with integer coefficients and
  degree $D$, which does not vanish on $G$, such that for all $\Delta\in
  \mathcal{D}(T/2)$ and all $\alpha\in \Gamma'(S)$ we have
  \begin{enumerate}
    \item[(a).] $\ord_{V,\alpha}P_\Delta \geq T/2$,
    \item[(b).] $h(P_\Delta) \leq c(D+T)\log(D+T) + cD(lS)^2$.
  \end{enumerate}
\end{prop}

To any polynomial $P$ satisfying the conditions given in the proposition and
any $\Delta\in \cD(\infty)$ we associate a function $f_\Delta =
\frac{P_\Delta}{c_\Delta X_N^{bD}}$, where $c_\Delta\in K$ is any coefficient
of $P_\Delta$ such that $\pabs{P_\Delta}=\pabs{c_\Delta}$.
\begin{lemma}\label{sec:auxiliary-polynomial}
  The functions $f_\Delta$ have the following properties:
  \begin{enumerate}
  \item[(a).] $\pabs{f_\Delta(\alpha)}\leq 1$ whenever $\pabs{X_i(\alpha)/X_N(\alpha)}\leq 1$
    for all $i$;
  \item[(b).] If $X_N(\alpha)\neq 0$, then $h(f_\Delta(\alpha))\leq
    \log\binom{N+bD}{N} + h(P_\Delta) + bDh(\alpha)$.
  \end{enumerate}
\end{lemma}
\begin{proof}
  Part (a) is trivial. Part (b) is easily shown using elementary height
  estimates.
\end{proof}

\begin{prop}[Upper bound]\label{sec:upper-bound-6}
  Let $\Delta\in\mathcal{D}(T/2)$, let $0\leq s \leq lS$, and let $P$ be
  the polynomial in Proposition~\ref{sec:auxiliary-polynomial-1}. Then
  \begin{equation*}
    \log \pabs{f_\Delta(\gamma_1^s)} \leq -cST
  \end{equation*}
\end{prop}
\begin{prop}[Lower bound]\label{sec:lower-bound-2}
  Let $\Delta\in \mathcal{D}(T/2)$, let $0\leq s \leq lS$ and let $P$ be the
  polynomial in Proposition~\ref{sec:auxiliary-polynomial-1}. Then
  either $P_\Delta(\gamma_1^s) = 0$, or
  \begin{equation*}
    \log\pabs{f_\Delta(\gamma_1^s)} \geq - c(lS)^2D - c(D+T)\log(D+T).
  \end{equation*}
\end{prop}
\begin{proof}[Proof of Theorem~\ref{sec:reduc-semist-pair}]
  Choose $l = S^\frac{3}{2}$, $T = S^{5n+2}$, $D =
  \floor{(S^{5nd+2d+2})^{1/n}}$. Then we can pick $S$ large enough so that
  the inequality \eqref{eq:2} is satisfied. Choose a polynomial $P$ with the
  properties given in Proposition~\ref{sec:auxiliary-polynomial-1}. Let
  $\Delta\in \cD(T/2)$, $0\leq s\leq lS$. Then, according to
  Proposition~\ref{sec:upper-bound-6}, for $S$ large enough we have
  \begin{equation}\label{eq:3}
      \log\pabs{f_\Delta(\gamma_1^s)} \leq -cST \leq -cS^{5n+3}
  \end{equation}
  On the other hand Proposition~\ref{sec:lower-bound-2} implies that either
  $P_\Delta(\gamma_1^s)=0$ or that
  \begin{equation*}
    \begin{split}
      -\log\pabs{f_\Delta(\gamma_1^s)} &\leq c(lS)^2D + c(D+T)\log(D+T) \\
      &\leq cS^{5 + 5d +\frac{2d+2}{n}} + c(S^{5d+\frac{2d+2}{n}} +
      S^{5n+2})\log S
    \end{split}
  \end{equation*}
  At this point we use the assumption that the linear space $V$ is a proper
  subspace of $\tG$. This means that $d\leq n-1$, hence we get the estimate
  \begin{equation}\label{eq:4}
    \begin{split}
      -\log\pabs{f_\Delta(\gamma_1^s)} &\leq cS^{5n+2}
      +c(S^{5n-2}+S^{5n+2})\log S\\
      &\leq cS^{5n+2.5}
    \end{split}
  \end{equation}
  However, if $S$ is large enough, the estimates \eqref{eq:3} and \eqref{eq:4}
  cannot simultaneously be satisfied, hence $P_\Delta(\gamma_1^s)=0$ for all
  non-negative $s\leq lS$ and all $\Delta\in\cD(T/2)$.

  Finally we apply Theorem~\ref{sec:differentiation-2} where we set $S_0:=lS$,
  $T_0:=T/2$, $D_0:=D$ and $\gamma_0:=\gamma_1$. Since $(lS)T^d \sim
  D^nS^\frac{1}{2}$, if we apply the theorem for large enough $S$ we get that
  $\gamma_1$ is a torsion point. This, however, contradicts the assumption
  that $\gamma$ is a non-torsion point, which completes the proof of the
  theorem.
\end{proof}
\section{Proofs of the main propositions}
\subsection{The auxiliary polynomial}
We are only going to give a sketch of the proof of Proposition~\ref{sec:auxiliary-polynomial-1}. It is very similar to the
proof of Lemma 4.1 in \cite{wustholz-89.2}, the only difference being that the
heights of the points in $\Gamma'(S)$ are bounded above by $c(lS)^2$ instead
of $cS^2$.  Condition (1) of the proposition is implied by having
$P_\Delta(\gamma'^s)=0$ for all $\Delta\in\cD(T)$ and all $s$, $0\leq s\leq
S$. Each of those equations is a linear equation for the coefficients of $P$ and
there are roughly $ST^d$ such equations. The space of polynomials of degree
$D$ modulo polynomials vanishing on $G$ has dimension roughly $D^n$. Therefore
if $D^n \geq cST^d$ there exists a polynomial satisfying Condition
(1). Siegel's Lemma tells us that we can pick the polynomial to have integer
coefficients and a certain upper bound of the height, determined by the height
of the system of linear equations. A not very difficult (but long) computation
shows that the height of the coefficients of each equation is bounded above
by $c(D+T)\log(D+T) + h(\gamma'^s)$. Using Lemma~\ref{sec:preliminaries} we
see that the second term is bounded above by $c(lS)^2$. Applying Siegel's Lemma
one shows the existence of a polynomial $P$ with small height which satisfies
Condition (1). It is then not difficult to show that condition (2) is also satisfied.

\subsection{The upper bound}
Let $f$ be a $v$-adic analytic function of one variable. We will use the
notation
\[
\norm{f}_r:=  \max_{\pabs{z}\leq r}{\pabs{f(z)}}.
\]

We shall call a power series $f(z)=\sum_{n=1}^{\infty}a_nz^n$ \term{normal},
if
\begin{enumerate}
  \item $\lim_{n\rightarrow\infty}\pabs{a_n} = 0$, and
  \item $\pabs{a_n} \leq 1$.
\end{enumerate}
Those condtions are equivalent to
\begin{enumerate}
\item $f(z)$ is defined and analytic on $\pabs{z}\leq
  1$.
\item $\norm{f}_1 \leq 1$.
\end{enumerate}

The following theorem is due to Mahler \cite{adams-66}*{Appendix, Theorem 14}:
\begin{thm}\label{sec:upper-bound-5}
  Let $f(z)$ be a normal function which has zeros at \\ \mbox{$x_1,\dots,x_n\in K_v$} of
  multiplicities $d_1,\dots,d_n$ respectively. Assume $\pabs{x_i} \leq r$, where
  $0<r<1$, and let $d = d_1+\dots+d_n$. Then given any
  $x\in K_v$ such that $\pabs{x}\leq r$ and any $k=0,1,2,\dots$ we
  have
  \[
  \pabs{f^{(k)}(x)} \leq r^{d-k}
  \]
\end{thm}
\begin{proof}[Proof of Proposition~\ref{sec:upper-bound-6}]
  Let $w:= \log\gamma$, and let $\phi(z):= f_\Delta(\exp(zw))$. (Here $\exp$
  is the inverse of restriction of the the logarithm function to $W_0$.) Since
  $\gamma\in W_0$ we have that $\phi$ is defined and analytic for all
  $\pabs{z}\leq 1$. For any $\alpha\in W_0$, since
  $\left\lvert\frac{X_i}{X_N}(\alpha)\right\rvert_v\leq 1$ for all $i$,
  Lemma~\ref{sec:auxiliary-polynomial}(a) implies that
  $\pabs{f_\Delta(\alpha)}\leq 1$. We conclude that $\norm{\phi}_1\leq 1$,
  hence $\phi$ is normal.
  
  By Proposition~\ref{sec:auxiliary-polynomial-1} (1), it follows that the
  order of $\phi$ at points \\ $0, pl, 2pl, \dots, plS$ is at least $T/2$. Let
  $r=\pabs{p}$. Then, since $0<r<1$ and $\pabs{spl}\leq r$, applying
  Theorem~\ref{sec:upper-bound-5} we conclude that for any integer $s$ we
  have
  \[
  \pabs{f_\Delta(\gamma_1^s)} = \pabs{\phi(ps)} \leq r^{ST/2}.
  \]
  Taking logarithms we obtain the desired result.
\end{proof}

\subsection{The lower bound}
\begin{proof}[Proof of Proposition~\ref{sec:lower-bound-2}]
  We are going to prove this proposition by means of Liouville's inequality
  \cite{hindry-silverman-00}*{Lemma~D.3.3}. Assume that
  $P_\Delta(\gamma_1^s)\neq 0$. Using
  Propositon~\ref{sec:auxiliary-polynomial-1},
  Lemma~\ref{sec:auxiliary-polynomial}(b) and Lemma~\ref{sec:preliminaries} we
  estimate
 \begin{equation*}
   \begin{split}
     h(f_\Delta(\gamma_1^s)) &\leq bD h(\gamma_1^s) + h(P_\Delta) + \log\binom{N+bD}{N}\\
     &\leq cs^2D + c(D+T)\log(D+T) + cD(lS)^2 + c\log D\\
     &\leq c(lS)^2D + c(D+T)\log(D+T)
   \end{split}
 \end{equation*}
 Since we have assumed that $P_\Delta(\gamma_1^s)\neq 0$, Liouville's inequality
 implies that 
 \[
 [K_v:\bQ_p]\log\pabs{f_\Delta(\gamma_1^s)} \geq -[K:\bQ]h(f_\Delta(\gamma_1^s))
 \]
 Combining this inequality with the previous estimate proves the assertion.
\end{proof}

\end{document}